\documentclass[a4paper,12pt]{amsart}
\usepackage[letterpaper, margin=1.2in]{geometry}
\usepackage{latexsym}
\usepackage{amssymb}
\usepackage{amsmath}
\usepackage{amsfonts}
\usepackage[table]{xcolor}
\usepackage{pgfplots}
\usepackage{color}

\usepackage{txfonts,pxfonts} 
\usepackage{tikz}
\usetikzlibrary{calc,arrows}
\usetikzlibrary{calc}
\usepackage{verbatim}

\newtheorem{thm}{Theorem}[section]

\newtheorem{lem}[thm]{Lemma}

\newtheorem{cor}[thm]{Corollary}

\theoremstyle{definition}

\newtheorem*{rem}{Remark}

\newtheorem*{examps}{Examples}

\def\fph{\mathbb{F}_{\ph}}

\newcommand{\Z}{\mathbb Z}
\newcommand{\z}{\mathbb Z}
\newcommand{\Q}{\mathbb Q}

\newcommand{\F}{\mathbb F}
\newcommand{\fp}{\mathbb F_p}

\def\F{\mathbb{F}}

\newcommand{\p}{\mathfrak{p}}

\def\al{\alpha}
\def\la{\lambda}

\def\th{\theta}
\def\md#1{\ \mbox{\rm(mod }{#1})}
\def\nph#1{N_{\ph}(#1)}
\def\npp#1{N_{\ph}^+(#1)}
\def\ph{\phi}

\newcounter{cs}
\stepcounter{cs}
\newcommand{\casos}{\begin{itemize}}
\newcommand{\fcasos}{\end{itemize}\setcounter{cs}{1}}

\newfont{\tit}{cmr12 scaled \magstep3}

\begin{document}
\title[]{On monogenity of certain  pure number fields defined by $x^{p^r}-m$}
\textcolor[rgb]{1.00,0.00,0.00}{}
\author{Hamid Ben Yakkou and Lhoussain El Fadil}
\address{Faculty of Sciences Dhar El Mahraz, P.O. Box  1794 Atlas-Fez , Sidi Mohamed Ben Abdellah University, Fez-- Morocco}\email{beyakouhamid@gmail.com\\lhouelfadil2@gmail.com}
\date{\today}
\maketitle
\begin{abstract}  
 Let $K = \mathbb{Q} (\alpha) $ be a pure number field generated by   a complex root $\alpha$  a monic irreducible polynomial  $ F(x) = x^{p^r} -m$, with $ m \neq 1 $  is a square free rational integer, $p$ is a rational prime integer,  and $r$ is a positive  integer.  In this paper, we study the monogenity of $K$. We prove that if {{$\nu_p(m^p-m)=1$}}, then $K$ is monogenic. But if $r\ge p$ and    {$\nu_p(m^{p}-m)> p$}, then $K$ is not monogenic.  Some illustrating examples are given.
\end{abstract}

\section{Introduction} 
Let $K$ be a number field defined by a monic irreducible polynomial $f(x)\in
  \Z[x]$ and $\Z_K$  its ring of integers.
  It is well know that  the ring $\Z_K$ is a free $\Z$-module of rank $n=[K:\Q]$. Thus the abelian group $\Z_K/\Z[\al]$ is finite. Its cardinal order is called the index of $\Z[\al]$, and denoted $(\Z_K: \Z[\al])$.
 Kummer showed  that for a rational prime integer, if $p$ does not divide $(\mathbb{Z}_K : \mathbb{Z}[\alpha])$, then the factorization of  $ p \mathbb{Z}_K$ can be derived directly  from the factorization of $\overline{F(x)}$ in $\F_p[x]$.  In $1878$, Dedekind  gave a criterion  to tests wither  $p$ divides or not  $(\Z_K: \Z[\al])$ (see  \cite[Theorem 6.1.4]{Co} and \cite{R}).
The index $ i(K) $ of the field $K$ is $ i(K) = \ gcd \ \{ ind (\theta) \ | \ \theta \in \mathbb{Z}_K \mbox{ generates } K \ \} $. A rational prime integer $p$ dividing $ i(K)$ is called a prime common index divisor of $K$.
  The ring $\Z_K$ is said to have a power integral basis if it has a $\Z$-basis $(1,\th,\cdots,\th^{n-1})$ for some $\th\in \Z_K$. That means $\Z_K=\Z[\th]$ or $\Z_K$ is mono-generated as a ring, with a single generator $\th$. In such  a case, the field $K$ is said to
be monogenic and not monogenic otherwise. 
 The problem of checking the monogenity of the field $K$ is  called a problem of Hasse \cite{F4, G19, Ha, He,  MNS}. 
The  problem of testing the monogenity  of number fields  and constructing  power integral bases have been intensively studied, this last century, mainly  by  Ga\'al, Nakahara, Peth\"o, and their research groups  (see for instance \cite{ AN,G, GO, P}). 
 { In \cite{E07},   El Fadil  gave conditions for the existence of power integral bases of pure cubic fields in terms of the index form equation. In \cite{F4}, Funakura, studied the integral basis in pure quartic fields.  In \cite{GR4},  Ga\'al and  Remete, calculated the elements of index $1$ {(with coefficients with absolute value $<10^{1000}$ in the integral basis)} of pure quartic fields generated by $m^{\frac{1}{4}}$ for {$1< m <10^7$}  and $m\equiv 2,3 \md4$.  In \cite{AN6}, Ahmad, Nakahara, and Husnine proved  that  if $m\equiv 2,3 \md4$ and  $m\not\equiv \mp1\md9$, then the sextic number field generated by $m^{\frac{1}{6}}$ is monogenic.
They also showed in \cite{AN},    that if $m\equiv 1 \md4$ and $m\not\equiv \mp1\md9$, then the sextic number field generated by $m^{\frac{1}{6}}$ is not monogenic.  
{Also, Hameed and Nakahara \cite{HN8}, proved that if $m\equiv 1\md4$, then the
octic number field generated by $m^{1/8}$ is not monogenic, but if $m\equiv 2,3 \md4$, 
then it is monogenic.}   In \cite{GR17}, Ga\'al  and  Remete, by applying the  explicit form of the index forms, they obtained new results on  monogenity of the number  fields generated  by $m^{\frac{1}{n}}$, where $3\le n\le 9$. } While Ga\'al's and Peth\"o's techniques are based on the index calculation,  Nakahara's methods are based on the existence of power relative integral bases of some special sub-fields. The goal of this paper is to study the monogenity of pure number fields of degree $p^n$ for every rational prime integer $p$ and $n$  is a natural integer. Our proposed results generalize  \cite[Theorem 2.1]{CNS}, where pure number fields of degree $2^n$ are previously studied. Our method is based on prime ideal factorization and is similar to that used in \cite{E6,E12, E24}. 
\section{Main results}
Let $K$ be the number field defined by a complex root $\al$ of   a monic  irreducible polynomial $F(x)=x^{p^r}-m$, with $m\neq 1$ is a square free rational integer.
\begin{thm}\label{pib}
 If $\nu_p(m^{p}-m)=1$, then $\Z[\al]$ is the ring of integers of $K${; $K$ is monogenic and $\al$ generates a power integral basis}.
\end{thm}
\begin{cor}
For $p=2$, if $m\equiv 2$ or $3$, then $\Z[\al]$ is integrally closed for every natural integer $r$.
 So that Theorem \ref{pib} generalizes  \cite[Theorem 2.1]{CNS}, where $p=2$ is  previously considered.
\end{cor}
\begin{thm}\label{npibodd}
 Assume that  $p$ is an odd prime integer and $p$ does not divide $m$. If {{$\nu_p(m^{p-1}-1)>p$ }}  and $r\ge p$, then $K$ is not monogenic.  
\end{thm}
\begin{cor}\label{mono3}
 Assume that $p=3$. 
 \begin{enumerate}
 \item
 If $m\not\equiv \mp 1\md9$, then $K$ is  monogenic. 
\item
 If $r\ge 3$ and $m\equiv \mp 1 \md{81}$, then $K$ is not monogenic. 
\end{enumerate}  
\end{cor}
\begin{thm}\label{mono2}
 Assume that $p=2$. \begin{enumerate}
 \item
 If $r=2$ and  $m\equiv 1\md{16}$, then the pure quartic number field $K$ is not monogenic. 
\item
 If $r\ge 3$ and  $m\equiv 1\md{32}$, then $K$ is not monogenic. 
\end{enumerate}  
\end{thm}
\section{Proofs}
 We start by recalling  some fundamental notions on  Newton polygon's techniques. {For more details, we refer to \cite{El, EMN}}. 
 For any  prime integer  $p$ and for any  monic polynomial 
$\phi\in
\z[x]$ {\it whose reduction} is irreducible  in
$\fp[x]$, let $\fph$ be 
the field $\frac{\fp[x]}{(\overline{\phi})}$. For any
monic polynomial  $f(x)\in \z[x]$, upon  the euclidean division
 by successive powers of $\ph$, we  expand $f(x)$ as
$f(x)=\sum_{i=0}^la_i(x)\phi(x)^{i}$, called    the $\phi$-expansion of $f(x)$
 (for every $i$, $deg(a_i(x))<
deg(\phi)$). 
The $\ph$-Newton polygon of $f(x)$ with respect to $p$, is the lower boundary convex envelope of the set of points $\{(i,\nu_p(a_i(x))),\, a_i(x)\neq 0\}$ in the euclidean plane, which we denote by $\nph{f}$. For every $i\ne j=0,\dots,l$, let  $a_i=a_i(x)$ and $\mu_{ij}=\frac{\nu_p(a_{i})-\nu_p(a_{j})}{i-j}\in \Q$.  Then we obtain the following integers $0=i_0<i_1<\dots< i_r=l$ satisfying   $i_{j+1}=\mbox{ max }\{i=i_j+1,\dots l,\, \mu_{i_ji_{j+1}} \le \mu_{i_ji}\}$.  For every $j=1,\dots r$, let {$S_j$ be the segment joining the points $A_{j-1}=(i_{j-1}, \nu(a_{i_{j-1}}))$  and $A_j=(i_{j}, \nu(a_{i_{j}}))$} in the euclidean plane. The segments $S_1,\dots,S_r$ are called the sides of the polygon $\nph{f}$.  For every $j=1,\dots,r$, the rational  number $\la_j=\frac{\nu_p(a_{i_{j}})-\nu_p(a_{i_{j-1}})}{i_{j}-i_{j-1}}\in \Q$ is called the slope of $S_j$, $l(S_j)=i_{j}-i_{j-1}$ is its length, and  $h(S_j)=-\la_jl(S_j)$ is its height. In what follows   $\nu(a_{i_{j}})=\nu(a_{i_{j-1}})+l(S_{j})\la_{j}$. The $\ph$-Newton polygon of $f$, is the process of joining the segments $S_1,\dots,S_r$ ordered by  the increasing slopes, which  can be expressed as $\nph{f}=S_1+\dots + S_r$.  Notice that $
\nph{f}=S_1+\dots+S_r$ is only a notation and  not the sum in the euclidean plane. 
 For every side $S$ of the polygon $\nph{f}$, $l(S)$ is 
 the length of its projection to the $x$-axis and  $h(S)$  is the length of its projection to the $y$-axis. 
 {The principal part of $\nph{f}$,
 denoted $\npp{f}$, is the part of the  polygon $\nph{f}$, which is  determined by joining all sides of negative  slopes.}

 For every side $S$ of $\nph{f}$, with initial point $(s, u_s)$ and length $l$, and for every 
$i=0, \dots,l$, we attach   the following
{{\ residual coefficient} $c_i\in\fph$ as follows:
$$c_{i}=
\left
\{\begin{array}{ll} 0,& \mbox{ if } (s+i,{\it u_{s+i}}) \mbox{ lies strictly
above } S
 ,\\
\left(\dfrac{a_{s+i}(x)}{p^{{\it u_{s+i}}}}\right)
\,\,
\md{(p,\phi(x))},&\mbox{ if }(s+i,{\it u_{s+i}}) \mbox{ lies on }S.
\end{array}
\right.$$
where $(p,\phi(x))$ is the maximal ideal of $\z[x]$ generated by $p$ and $\ph$. That means if $(s+i,{\it u_{s+i}}) \mbox{ lies on }S$, then $c_i=\overline{\dfrac{a_{s+i}(\beta)}{p^{{\it u_{s+i}}}}}$, where $\beta$ is a root of $\ph$.}

Let  $\la=-h/e$ be the slope of $S$, where  $h$ and $e$ are two positive coprime integers, and let $d=l/e$ be the degree of $S$.  Notice that, 
the points  with integer coordinates lying in $S$ are exactly $(s,u_s),(s+e,u_{s}+h),\cdots, (s+de,u_{s}+dh)$. Thus, if $i$ is not a multiple of $e$, then 
$(s+i, u_{s+i})$ does not lie in $S$, and so, $c_i=0$. Let
$f_S(y)=t_0y^d+t_1y^{d-1}+\cdots+t_{d-1}y+t_{d}\in\fph[y]$, called  
the residual polynomial of $f(x)$ associated to the side $S$, where for every $i=0,\dots,d$,
 $t_i=c_{ie}$.
\begin{rem}
 Note that if $\nu(a_{s}(x))=0$, $\la=0$, and $\ph=x$, then $\fph=\F_p$ and for every $i=0,\dots,l$, $c_i=\overline{{a_{s+i}}} \md{p}$. Thus this notion of residual coefficient generalizes the reduction modulo a maximal ideal and $f_S(y)\in\F_p[y]$ coincides with the reduction of $f(x)$ modulo the maximal ideal $(p)$.
    \end{rem}
    Let $\npp{f}=S_1+\dots + S_r$ be the { principal} $\ph$-Newton polygon of $f$ with respect to $p$. We say that $f$ is a $\ph$-regular polynomial with respect to $p$, if for every  $i=1,\dots,r$, $f_{S_i}(y)$ is square free in $\fph[y]$. We say that $f$ is a $p$-regular polynomial if $f$ is  a $\ph_i$-regular polynomial with respect to $p$ for every $i=1,\dots,t$, for some monic polynomials $\ph_1,\dots,\ph_t$ with  $\overline{f(x)}=\prod_{i=1}^t\overline{\ph_i}^{l_i}$ is the factorization of $\overline{f(x)}$ in $\F_p[x]$.

The  theorem of Ore plays a fundamental key for proving our main Theorems:\\
  Let $\ph\in\Z[x]$ be a monic polynomial, with $\overline{\ph(x)}$ is irreducible in $\F_p[x]$. As defined in \cite[Def. 1.3]{EMN},   the $\ph$-index of $f(x)$, denoted by $ind_{\ph}(f)$, is  deg$(\ph)$ times the number of points with natural integer coordinates that lie below or on the polygon $\npp{f}$, strictly above the horizontal axis, {and strictly beyond the vertical axis} (see Figure 1).

\begin{figure}[htbp]

\centering

\begin{tikzpicture}[x=1cm,y=0.5cm]
\draw[latex-latex] (0,6) -- (0,0) -- (10,0) ;

\draw[thick] (0,0) -- (-0.5,0);
\draw[thick] (0,0) -- (0,-0.5);

\node at (0,0) [below left,blue]{\footnotesize $0$};

\draw[thick] plot coordinates{(0,5) (1,3) (5,1) (9,0)};
\draw[thick, only marks, mark=x] plot coordinates{(1,1) (1,2) (1,3) (2,1)(2,2)     (3,1)  (3,2)  (4,1)(5,1)  };

\node at (0.5,4.2) [above  ,blue]{\footnotesize $S_{1}$};
\node at (3,2.2) [above   ,blue]{\footnotesize $S_{2}$};
\node at (7,0.5) [above   ,blue]{\footnotesize $S_{3}$};
\end{tikzpicture}
\caption{    \large  $\npp{f}$.}
\end{figure}

 Now assume that $\overline{f(x)}=\prod_{i=1}^t\overline{\ph_i}^{l_i}$ is the factorization of $\overline{f(x)}$ in $\F_p[x]$, where every $\ph_i\in\Z[x]$ is monic, $\overline{\ph_i(x)}$ is irreducible in $\F_p[x]$, $\overline{\ph_i(x)}$ and $\overline{\ph_j(x)}$ are coprime when $i\neq j=1,\dots,t$.
For every $i=1,\dots,t$, let  $N_{\ph_i}^+(f)=S_{i1}+\dots+S_{ir_i}$ be the principal part of the $\ph_i$-Newton polygon of $f$ with respect to $p$. For every {$j=1,\dots, r_i$},  let $f_{S_{ij}}(y)=\prod_{s=1}^{s_{ij}}\psi_{ijk}^{a_{ijs}}(y)$ be the factorization of $f_{S_{ij}}(y)$ in $\F_{\ph_i}[y]$. 
  Then we have the following  theorem of Ore (see \cite[Theorem 1.7 and Theorem 1.9]{EMN}, \cite[Theorem 3.9]{El}, and \cite{O}):
 \begin{thm}\label{ore} (Theorem of Ore)
 \begin{enumerate}
 \item
  $\nu_p(ind(f))=\nu_p((\z_K:\z[\al]))\ge \sum_{i=1}^t ind_{\ph_i}(f)$ and  equality holds if $f(x)$ is $p$-regular;  every $a_{i,j,s}= 1$.
\item
If  every $a_{i,j,s}= 1$, then
$$p\Z_K=\prod_{i=1}^r\prod_{j=1}^{r_i}
\prod_{s=1}^{s_{ij}}\p^{e_{ij}}_{i,j,s},$$ where $e_{ij}$ is the ramification index
 of the side $S_{ij}$ and $f_{ijs}=\mbox{deg}(\ph_i)\times \mbox{deg}(\psi_{ijs})$ is the residue degree of $\p_{ijs}$ over $p$.
 \end{enumerate}
\end{thm}
\begin{cor}\label{ind}
 Under the hypothesis of Theorem \ref{ore}, if for every $i=1,\dots,t$, $l_i=1$ or $N_{\ph_i}(f)=S_i$ has a single side of height $1$, then $\nu_p(ind(f))=0$.
\end{cor}
The following lemma allows to evaluate  the $p$-adic valuation of the binomial coefficient $\binom{p^r}{j} $. 
\begin{lem} \label{binomial}
	Let $p$ be a rational prime integer and $r$ be a positive integer. Then $ \nu_p\left(\binom{p^r}{j}\right)  =  r - \nu_p(j)$
	for any integer $j= 1,\dots,p^r-1 $. 
\end{lem}
\begin{proof}
Since for any natural number $m$, $\nu_p( m!)= \sum_{t=1}^{m} \nu_p(t)$, 
we have $ \nu_p \left( \binom{p^r}{j} \right) = \nu_p(p^r ! ) - \nu_p((p^r-j)!) -\nu_p(j!) =  \sum_{t=1}^{p^r} \nu_p(t) - \sum_{t=1}^{p^r - j } \nu_p(t) - \sum_{t=1}^{j} \nu_p(t)$. Thus 
 $ \nu_p \left( \binom{p^r}{j} \right)
 = \sum_{t= p^r - j +1}^{p^r } \nu_p(t) - \sum_{t=1}^{j} \nu_p(t) = \nu_p(p^r) + \sum_{t=1}^{j-1} \nu_p(p^r - t) - \sum_{t=1}^{j} \nu_p(t) $. 
As for every $ t=1,\dots, j-1 < p^r $, then $\nu_p(p^r - t ) = \nu_p(t)$. Hence $ \nu_p    \left(  \binom{p^r}{j} \right) = r - \nu_p(j)$.
\end{proof}
\begin{proof} of Theorem \ref{pib}.\\
 Since $\triangle(f)  =\mp p^{rp^r} m ^{p^r - 1}$, then by the  formula linking $\triangle(f)$, $ (\Z_K: \Z[\al])$ and $d_K$, $\mathbb{Z}[\alpha]$ is integrally closed if and only if   $q$ does not divide $(\Z_K: \Z[\al])$ for every rational prime integer $q$ dividing $pm$.  Let $q$ be a rational prime dividing $m$, then $ f(x) \equiv \phi^{p^r}  (\ mod \ q )$, where $\phi = x $. As $m$ is a square free integer, the $\phi$-principal Newton polygon with respect to $\nu_q$,  $ \npp {f} = S $ has a single side of degree $1$. Thus $f_{S}(y)$ is irreducible over $\mathbb{F}_\phi \simeq \mathbb{F}_q$. By Theorem \ref{ore}, we get $\nu_q((\Z_K: \Z[\al])) = ind_{\ph}(f) = 0 $; $q$ does not divide $(\Z_K: \Z[\al])$. Now, we deal with $p$, when $p$ does not divide $m$. In this case, let $ \nu = v_p(m^{p-1} - 1)$.  Since
 $ f(x) = ( x - m+m)^{p^r} - m =(x-m)^{p^r}   + \sum_{j=1}^{p^r - 1} \binom{p^r}{j} m^{p^r - j} (x-m)^j + m^{p^r}  - m $, then by Lemma \ref{binomial},  $ f(x) \equiv \phi^{p^r}  ( mod \ p )$, where $\phi = x -m$. If $ \nu =1$,  then $\npp {f} = S $ has a single side of height $1$ and by Theorem \ref{ore}, $ p $ does not divide $ (\Z_K: \Z[\al])$.  But if $\nu\ge 2$, then  $ p $  divides $(\Z_K: \Z[\al])$. 
\end{proof} 
 For the proof of Theorem \ref{npibodd}, Corollary \ref{mono3}, and Theorem \ref{mono2}, we need the following Lemma, which is an immediate consequence of a theorem of Kummer.
 \begin{lem} \label{comindex}
 Let  $p$ be  rational prime integer and $K$ be a number field. For every positive integer $f$, let $P_f$ be the number of distinct prime ideals of $\Z_K$ lying above $p$ with residue degree $f$ and $N_f$ be the number of monic irreducible polynomials of  $\F_p[x]$ of degree $f$.
  If $ P_f > N_f$, then $p$ is a common  index divisor of $K$.
\end{lem}
 \begin{proof} of Theorem \ref{npibodd}.\\
 Under the hypothesis of Theorem \ref{npibodd}, $\overline{F(x)}=(x-m)^{p^r}$ in $\F_p[x]$ and $ F(x) = ( x - m+m)^{p^r} - m =(x-m)^{p^r}   + \sum_{j=1}^{p^r - 1} \binom{p^r}{j} m^{p^r - j} (x-m)^j + m^{p^r}  - m $.  Let $\phi = x -m$ and $\nu=\nu_p (m^{p^r}-m)$. As $m^{p^r}-m=(m^{p}-m)(m^t+\dots+1)$, where $t=p^r-1$,  if $m\equiv 1 \md{p^{\nu}}$, then $m^t+\dots+1\equiv -1\md{p}$, and so, $\nu_p(m^t+\dots+1)=0$. Thus $\nu_p(m^{p-1}-1)=\nu_p(m^{p}-m)=\nu_p(m^{p^r}-m)$. It follows that if $ \nu >p$ and $r\ge p$, then by Lemma \ref{binomial}, $\npp{f}= S_1+\dots+S_{t-p+1}+\dots+S_{t}$ has $t$-distinct sides of degree $1$  each one, with $t\ge p+1$. More precisely, $S_{t-i}$ has $(p^{r-i-1},i+1)$ as the first point and $(p^{r-i},i)$ as the last point for every $i=0,\dots,p-1$,  and $S_1$ has $(0, \nu )$ as the first point and $(1,r)$ or $ (p^{r-\nu+1}, \nu - 1)$ as the last point according to $\nu > r $ or  $ r \geq \nu$.  It follows  by Theorem \ref{ore} that, there are $t$ distinct prime ideals of $\Z_K$ lying above $p$,with residue degree $1$ each one. As {{ $t = P_1 > p=N_1$}}, by Lemma \ref{comindex}, $p$ is a common index divisor of $K$ and $K$ is not monogenic.
   \end{proof}
   \begin{proof} of Theorem \ref{mono2}.\\
   Let $m$ be an odd rational integer. Then  $\overline{F(x)}=(x-1)^{p^r}$ in $\F_p[x]$ and $ F(x) = ( x - 1+1)^{p^r} - m =(x-1)^{p^r}   + \sum_{j=1}^{p^r - 1} \binom{p^r}{j}  (x-1)^j + 1 - m $.  Let $\phi = x -1$ and $\nu=\nu_p (1-m)$. It follows that:
   \begin{enumerate}
   \item
   If $r=2$ and $\nu\ge 4$, then by Lemma \ref{binomial}, {$\npp{f}=S_1+S_2+S_3$}  has $3$-distinct sides, with degree $1$ each side. Thus there are $3$-distinct prime ideals of $\Z_K$ lying above $2$ with residue degree $1$ each one. So, by Lemma \ref{comindex}, $2$ is a common index divisor of $K$.
   \item
   If $r\ge 3$ and $\nu\ge 5$, then by Lemma \ref{binomial}, $\npp{f}=S_1+S_2+\dots+S_t$  has  $t$-distinct sides, with $t\ge 4$ and  degree $S_i$ equals $1$  for every  $i=2,\dots,t$. Thus there are at least $t-1$-distinct prime ideals of $\Z_K$ lying above $2$ with residue degree $1$ each one. As $t-1\ge 3$, then by Lemma \ref{comindex}, $2$ is a common index divisor of $K$.
      \end{enumerate}
    \end{proof}
\begin{examps} 
	\begin{enumerate}
	\item 
	{ If $ \overline{m} \not\in\{ \overline{1}, \overline{18}, \overline{19}, \overline{30}, \overline{31}, \overline{48}\} \md { 49} $, then the pure number field   $\mathbb{Q}  (\sqrt[7^r]{m} ) $ is monogenic for every positive integer $r$. }
\item { If $ m^{p-1} \equiv 1  \md { p^{p+1}}$, then for every natural integer $r\geq p$ the pure field $\mathbb{Q}  ({\alpha} ) $ is not monogenic.  In particular, if $m^4\equiv 1 \md{5^6}$, then the pure field $\mathbb{Q}  (\sqrt[5^r]{m} )  $ is not monogenic, for every natural integer $r\geq  5$.}
\item { If $m\equiv 1\md{32}$, then the octic number field  $\mathbb{Q}   (\sqrt[8]{m} ) $ is not monogenic  }.
\item  More general, if $m\equiv 1\md{32}$, then for every natural integer $r\ge2$, $\mathbb{Q}  ({\alpha} ) $ is not monogenic, where $\alpha$ is a complex root of $x^{2^r}-m$.
\item { If $\nu_3(m^2-1) \geq 4 $, the field  $\mathbb{Q}  (\sqrt[27]{m} ) $ is not monogenic. ( see figure $2$ )}.
\item  { More general, if $m\equiv \mp 1\md{81}$, then the field  $\mathbb{Q}  ({\alpha} ) $ is not monogenic, where $\alpha$ is a complex root of $x^{3^r}-m$ and $r\ge 3$ is a natural integer.}
 \end{enumerate}
		\end{examps}
\centering
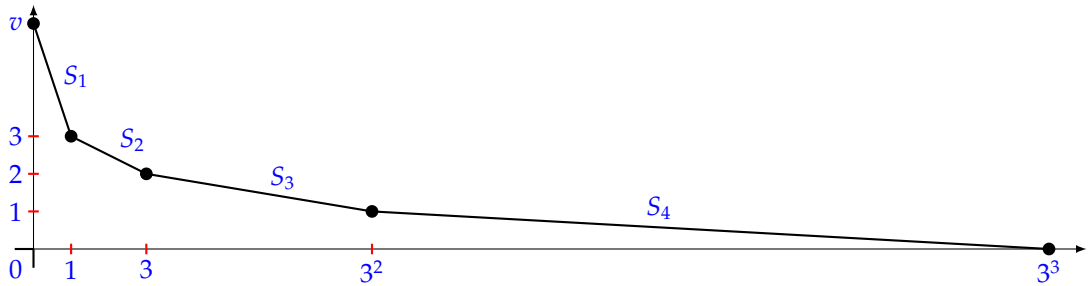
\begin{figure}[htbp] 
\begin{tikzpicture}[x=0.5cm,y=0.5cm]
\draw[latex-latex] (0,6.5) -- (0,0) -- (28,0) ;
\draw[thick] (0,0) -- (-0.5,0);
\draw[thick] (0,0) -- (0,-0.5); 
\draw[thick,red] (1,-2pt) -- (1,2pt);
\draw[thick,red] (3,-2pt) -- (3,2pt);
\draw[thick,red] (9,-2pt) -- (9,2pt);
\draw[thick,red] (-2pt,1) -- (2pt,1);
\draw[thick,red] (-2pt,2) -- (2pt,2);
\draw[thick,red] (-2pt,3) -- (2pt,3);
\node at (0,0) [below left,blue]{\footnotesize  $0$};
\node at (1,0) [below ,blue]{\footnotesize  $1$};
\node at (3,0) [below ,blue]{\footnotesize $3$};
\node at (9,0) [below ,blue]{\footnotesize  $3^{2}$};
\node at (27,0) [below ,blue]{\footnotesize  $3^{3}$};
\node at (0,1) [left ,blue]{\footnotesize  $1$};
\node at (0,2) [left ,blue]{\footnotesize  $2$};
\node at (0,3) [left ,blue]{\footnotesize  $3$};
\node at (0,6) [left ,blue]{\footnotesize  $v$};
\draw[thick,mark=*] plot coordinates{(0,6) (1,3)(3,2) (9,1) (27,0)};
\node at (0.5,4) [above right  ,blue]{\footnotesize  $S_{1}$};
\node at (2,2.3) [above right  ,blue]{\footnotesize  $S_{2}$};
\node at (6,1.3) [above right  ,blue]{\footnotesize  $S_{3}$};
\node at (16,0.5) [above right  ,blue]{\footnotesize  $S_{4}$};
\end{tikzpicture}
\caption{    \large  The $(x-m)$-Newton polygon for $F(x)=x^{27}-m$  at $p=3$  and $v\geq 4$.\hspace{5cm}}
\end{figure}


\begin{thebibliography}{99}
\bibitem{AN6} {\sc S. Ahmad, T. Nakahara, and S. M. Husnine}, {\em Power integral bases for certain pure sextic
fields}, Int. J. of Number Theory v:10, No 8 (2014) 2257-- 2265. 
\bibitem{CNS}  A. Hameed, T. Nakahara, S. M. Husnine, On existence of canonical number system in certain classes of pure algebraic number fields, J. Prime Res. Math. 7(2011) 19-24. 
\bibitem{AN} {\sc S. Ahmad, T. Nakahara, and A. Hameed}, {\em On certain pure sextic fields related to a problem of Hasse}, Int. J. Alg. and Comput. 26(3) (2016) 577--583 
\bibitem{HN8} {\sc A.Hameed  and  T.Nakahara}, {\em Integral  bases  and  relative  monogenity  of pure octic fields}, Bull. Math. Soc. Sci. Math. R épub. Soc. Roum. 58(106) No. 4(2015) 419--433

\bibitem{Co} H. Cohen, {\em A Course in Computational Algebraic Number Theory}, GTM 138, Springer-Verlag Berlin  Heidelberg (1993)
\bibitem {E07} L. El Fadil, {\em Computation of a power integral basis of a pure cubic number
field}, Int. J. Contemp. Math. Sci.
2(13-16)(2007) 601--606
\bibitem {E6} L. El Fadil, {\em On Power integral bases for certain pure sextic fields} (To appear in the forthcoming issue of Bol. Soc. Paran. Math.)
\bibitem {E12} L. El Fadil, {\em On Power integral bases for certain pure number fields } (preprint)
\bibitem {E24} L. El Fadil, {\em On Power integral bases for certain pure number fields defined by $x^{24}-m$ } (To appear in the forthcoming issue of Stud. Sci. Math. Hung.)
\bibitem{El} L. El Fadil,\emph{ On Newton polygon's techniques and factorization of polynomial over henselian valued fields},  J. of Algebra and  its Appl.  (2020), doi: S0219498820501881
\bibitem{EMN}{\sc L. El Fadil}, {\sc J. Montes} and 
{\sc E. Nart}, {\em Newton polygons and $p$-integral bases of quartic number fields}, 
J. Algebra and Appl. 11(4)(2012) 1250073
\bibitem{F4} {\sc T. Funakura}, {\em On integral bases of pure quartic fields}, Math. J.
Okayama Univ. 26 (1984) 27-–41
\bibitem{G} {\sc  I. Ga\'al}, {\em Power integral bases in algebraic number fields}, Ann. Univ. Sci. Budapest. Sect. Comp. 18 (1999) 61--87 
\bibitem{G19} {\sc I. Ga\'al}, Diophantine equations and power integral bases, Theory and algorithm, Second edition, Boston, Birkh\"auser, 2019
\bibitem{GO} {\sc  I. Ga\'al, P. Olajos, and M. Pohst}, {\em Power integral bases in orders of composite fields}, Exp. Math. 11(1) (2002) 87–90. 
\bibitem{GR4} {\sc  I. Ga\'al and L. Remete}, {\em Binomial Thue equations and power integral bases
in  pure  quartic  fields}, JP  Journal  of  Algebra  Number  Theory  Appl. 32(1)
(2014) 49--61
\bibitem{GR17} {\sc  I. Ga\'al and L. Remete}, {\em Power integral bases and monogenity of pure  fields}, J. of Number Theory 173 (2017) 129--146 

\bibitem{Ha} {\sc H. Hasse}, {\em Zahlentheorie}, Akademie-Verlag, Berlin, 1963.
\bibitem{He} {\sc K. Hensel}, {\em Theorie der algebraischen Zahlen}, Teubner Verlag, LeipzigBerlin, 1908.
\bibitem{MNS} Y. Motoda, T. Nakahara and S. I. A.
Shah,  {\em On a problem of Hasse}, J. Number Theory 96 (2002)
326–-334
%\bibitem{Na} {\sc W. Narkiewicz}, {\em Elementary and Analytic Theory of Algebraic Numbers}, Third Edition, Springer, 2004.

\bibitem{O}{\sc O. Ore}, {\em Newtonsche
Polygone in der Theorie der algebraischen Korper}, Math. Ann., 99
(1928), 84--117
\bibitem{P}{\sc A. Peth\"o and M. Pohst},  {\em On the indices of multiquadratic number fields}, Acta Arith. 153(4) (2012) 393--414
\bibitem{R} R. Dedekind,  {\em \"Uber den Zusammenhang zwischen der Theorie der Ideale und der Theorie der h\"oheren Kongruenzen}, G\"ottingen Abhandlungen {\bf 23} (1878) 1--23
\end{thebibliography}
\end{document}